\documentclass[11pt,reqno]{amsart}

\usepackage[T1]{fontenc}
\usepackage{lmodern}
\usepackage[letterpaper,margin=1in]{geometry}
\usepackage{amssymb,mathtools}
\usepackage[hidelinks]{hyperref}
\usepackage[capitalize]{cleveref}

\newtheorem{theorem}{Theorem}
\newtheorem{proposition}[theorem]{Proposition}
\newtheorem{lemma}[theorem]{Lemma}
\theoremstyle{definition}
\newtheorem{remark}[theorem]{Remark}
\newtheorem*{acknowledgements}{Acknowledgements}

\newcommand{\C}{\mathbb{C}}
\newcommand{\R}{\mathbb{R}}
\newcommand{\cH}{\mathcal{H}}
\newcommand{\pbar}{\partial_{\bar z}}
\DeclareMathOperator{\Tr}{Tr}

\title{Planar Dirichlet eigenvalue bounds from a trace identity}
\author{Romain Speciel}
\address{Department of Mathematics, Stanford University, 450 Jane Stanford Way, Building 380, Stanford, CA 94305-2125, USA}
\email{rspeciel@stanford.edu}
\date{October 4, 2026}

\begin{document}

\begin{abstract}
	We prove an asymptotically sharp lower bound for the eigenvalues $\lambda_k$ of the Dirichlet Laplacian on bounded open sets in the plane with area $A>0$. The proof employs a trace identity for the commutator between a certain multiplication operator and its adjoint. As a corollary, we obtain $(2/3)\cdot 4\pi k/A\leq \lambda_k$ for every $k$, which is two thirds of the bound conjectured by P\'olya.
\end{abstract}

\maketitle

Fix a bounded open set $\Omega\subset\R^2$ with area $A>0$, and consider the Dirichlet Laplacian $\Delta$ on $\Omega$ with eigenvalues $0<\lambda_1\leq\lambda_2\leq\dots\to\infty$. In 1954, P\'olya conjectured \cite[pp.~51--53]{PolyaConjecture} that
\begin{equation}
	\frac{4\pi k}{A}\leq\lambda_k,
\end{equation}
and, in 1961, he went on to prove this bound for domains which tile the plane \cite{Polya}. When initially stating the problem, P\'olya added \cite[p.~52]{PolyaConjecture}
\begin{quote}
    ``This is a challenging conjecture. As the shape of the region [$\Omega$] can arbitrarily vary and [$k$] ranges over all the integers $1, 2, 3, \dots$, the conjecture covers an immense variety of particular cases of which not too much is known: the conjecture could be exploded by a numerical result concerning any of these particular cases.''
\end{quote}
The final clause of this remark makes this problem a particularly attractive candidate for recently-in-fashion AI-assisted counterexample searches \cite{EpochAI}. Nevertheless, it has thus far withstood such scrutiny. Recent results by Filonov, Levitin, Polterovich and Sher have in fact verified the conjecture for discs and circular sectors \cite[Theorems~1.2 and~1.12]{FLPSBalls} and for annuli \cite[Theorem~1.1]{FLPSAnnuli}, though the general case remains open.

If P\'olya's conjecture holds, the average of the first $k$ eigenvalues, which we denote by $m_k=(\lambda_1+\dots+\lambda_k)/k$, would satisfy $4\pi k/A\leq2m_k$.
This weaker estimate, known as the Berezin--Li--Yau inequality, was proven in \cite[Theorem~1]{LiYau}. In particular, since $m_k\leq\lambda_k$, it gives $ c\cdot 4\pi k/A\leq \lambda_k$ with $c=1/2$, whereas P\'olya's conjecture asserts this with $c=1$. Wang and Xie recently showed that one may take $c=0.5383\ldots$ \cite[Theorem~4.2]{WangXie}. In this note, we shall prove the following general estimate, from which it follows that the lower bound holds with $c=2/3$.

\begin{theorem}\label{thm}
	Let $\Omega\subset \R^2$ be a bounded open set with area $A>0$. Denote by $\lambda_k$ the $k$th eigenvalue of the Dirichlet Laplacian and by $m_k$ the average of the first $k$ such eigenvalues. For every $k\geq1$, we have
	\begin{equation}\label{eq:evalue_bound}
		\frac{2m_k}{4m_k-\lambda_k}\cdot\frac{4\pi k}{A}\leq\lambda_k.
	\end{equation}
\end{theorem}

\noindent By an inequality of Yang \cite[(1.10)]{Ashbaugh}, the left-hand side of \cref{eq:evalue_bound} is always positive. For completeness, we shall prove this in \cref{lem:partial}. Since Weyl's law \cite[Theorem~3.3.1 and Remark~3.3.2]{LMP} implies $m_k\sim\lambda_k/2$, the factor $2m_k/(4m_k-\lambda_k)$ tends to $1$ as $k\to \infty$, hence the bound is asymptotically sharp. Furthermore, since $m_k\leq\lambda_k$ trivially, it follows that
\begin{equation}
	\frac23\cdot\frac{4\pi k}{A}\leq\lambda_k
\end{equation}
for all $k$, as claimed. Note also that, after a simple rearrangement, one sees that \cref{thm} implies the Berezin--Li--Yau inequality. The proof below may therefore serve as an almost-elementary alternative, in the planar case, to the proofs of the Berezin--Li--Yau inequality that appear in the present literature.

We begin with some definitions. Identify $\R^2\cong\C$ by $z=x+iy$ and consider all function spaces over $\mathbb{C}$. Equip the space $\cH=H^1_0(\Omega)$ with the inner product
\[
	\langle u,v\rangle_{\cH}=\int_\Omega\nabla u\cdot\overline{\nabla v}\,dA,
\]
inducing the norm $\|\cdot\|_{\mathcal{H}}$. Equivalently, this definition may be stated in terms of the standard inner product $\langle\cdot,\cdot\rangle$ and norm $\|\cdot\|$ on $L^2(\Omega)$ by setting $\pbar=\frac{1}{2}(\partial_x+i\partial_y)$ and noting that, after integrating by parts,
\begin{equation}\label{eq:dbar}
	\langle u,v\rangle_{\cH}=4\int_\Omega\pbar u\,\overline{\pbar v}\,dA=4\langle \pbar u, \pbar v\rangle \quad\textrm{so}\quad \|u\|_\mathcal{H}=2\|\pbar u\|.
\end{equation}
Since $z$ is holomorphic, we have $\pbar(zu)=z\,\pbar u$, hence the operator $M\colon\cH\to\cH$ defined by $Mu=zu$ satisfies
\[
    \|Mu\|_\cH=2\|z\,\pbar u\|\leq\Bigl(\sup_\Omega|z|\Bigr)\|u\|_\cH
\]
and is therefore bounded. The key property of $M$, stated below, follows from a more general result of Berger and Shaw, as discussed in \cref{rem:bergman}. We shall give here a simple spectral proof.

\begin{proposition}[{cf.~\cite[Theorem~7]{BergerShaw}}]\label[proposition]{prop}
	The commutator $C=[M^*,M]$ is trace class, with
	\begin{equation}
		\Tr C=\frac{A}{\pi}.
	\end{equation}
\end{proposition}

To briefly elucidate why one might expect \cref{prop} to yield spectral bounds, fix real-valued eigenfunctions $u_1, u_2, \dots$ with $\Delta u_j=\lambda_ju_j$ forming an orthonormal basis of $L^2(\Omega)$, and note that the functions $e_j=u_j/\sqrt{\lambda_j}$ form an orthonormal basis of $\mathcal{H}$. Now, on $L^2(\Omega)$, multiplication by $z$ is normal and its matrix entries $c_{jk}=\langle zu_j,u_k\rangle=c_{kj}$ are symmetric. In other words, multiplying by $z$ shifts as much mass from $u_j$ to $u_k$ as from $u_k$ to $u_j$. However, in $\cH$, this symmetry is broken since
\begin{equation}\label{eq:matrix_for_M}
    \langle Me_j,e_k\rangle_\cH=\sqrt{\frac{\lambda_k}{\lambda_j}}\,c_{jk}
    \qquad\text{and}\qquad
    \langle M^*e_j,e_k\rangle_\cH=\sqrt{\frac{\lambda_j}{\lambda_k}}\,\overline{c_{jk}}.
\end{equation}
Therefore, $M$ places more weight on shifts to higher eigenvalues and less on shifts to lower ones, while $M^*$ does the reverse. The diagonal entries $a_j=\langle Ce_j,e_j\rangle_\cH=\|Me_j\|_\cH^2-\|M^*e_j\|_\cH^2$ measure the resulting imbalance at the $j$th eigenfunction. Summing over the first $k$ eigenfunctions results in cancellations and what remains is the net upward spectral shift induced by $M$ across $\lambda_k$. By closely studying this quantity, we obtain \cref{thm}.

We make this intuition precise through the several lemmas below. To this end, introduce the coefficients $d_{jk}=(\lambda_k-\lambda_j)\,|c_{jk}|^2=-d_{kj}$.

\begin{lemma}\label[lemma]{lem:positive}
	For every $u\in\cH$, we have $\|M^*u\|_\cH\leq\|Mu\|_\cH$, hence $C$ is a nonnegative operator. Furthermore, $a_j=\langle Ce_j,e_j\rangle_\cH\geq0$ may be written as
\begin{equation}
a_j=\sum_k\Bigl(\frac1{\lambda_j}+\frac1{\lambda_k}\Bigr)d_{jk}.
\end{equation}
\end{lemma}

\begin{proof}
	For $u,v\in\cH$, \cref{eq:dbar} gives
	\[
		|\langle M^*u,v\rangle_{\cH}|=4\,|\langle\bar z\,\pbar u,\pbar v\rangle|\leq2\|z\,\pbar u\|\cdot2\|\pbar v\|=\|Mu\|_{\cH}\|v\|_{\cH},
	\]
	from which it follows that $\|M^*u\|_{\cH}\leq\|Mu\|_{\cH}$. Conclude
	\[
		\langle Cu,u\rangle_{\cH}=\|Mu\|_{\cH}^2-\|M^*u\|_{\cH}^2\geq0,
	\]
	so $C$ is nonnegative. The second part of the statement follows from \cref{eq:matrix_for_M} since
	\begin{equation}\label{eq:h}
		a_j=\|Me_j\|_{\cH}^2-\|M^*e_j\|_{\cH}^2
		=\sum_k\Bigl(\frac{\lambda_k}{\lambda_j}-\frac{\lambda_j}{\lambda_k}\Bigr)|c_{jk}|^2
		=\sum_k\Bigl(\frac{1}{\lambda_k}+\frac{1}{\lambda_j}\Bigr)d_{jk},
	\end{equation}
	concluding the proof.
\end{proof}

Next, we record two identities for the coefficients $d_{jk}$.

\begin{lemma}\label[lemma]{lem:rules}
	The following identities hold for any fixed $j$:
	\begin{equation}\label{eq:rules}
	\begin{gathered}
		\sum_k d_{jk}=\|zu_j\|_{\cH}^2-\lambda_j\|zu_j\|^2
		=2,\\
		\sum_k(\lambda_k+\lambda_j)\,d_{jk}=\|\Delta(zu_j)\|^2-\lambda_j^2\|zu_j\|^2
		=8\lambda_j.
	\end{gathered}
	\end{equation}
\end{lemma}

\begin{proof}
	By \cref{eq:dbar}, we have
	\[
		\langle zu_j,v\rangle_\cH=4\langle\pbar u_j,\pbar(\bar zv)-v\rangle=\langle\lambda_jzu_j-4\pbar u_j,v\rangle
	\]
	for every $v\in\cH$. Deduce that $zu_j$ lies in the domain of $\Delta$, with $\Delta(zu_j)=\lambda_jzu_j-4\pbar u_j$. We may therefore rewrite $zu_j$ as $\sum_k c_{jk}u_k$ to then obtain
	\[
		\langle (\Delta-\lambda_j)(zu_j),(zu_j)\rangle=\sum_k d_{jk}.
	\]
	Rewriting the left-hand side as $\langle\Delta(zu_j)-\lambda_jzu_j,zu_j\rangle=-4\langle\pbar u_j,zu_j\rangle$ and computing
	\[
		\langle\pbar u_j,zu_j\rangle=\frac12\int_\Omega\bar z\,\pbar(u_j^2)\,dA=-\frac12\int_\Omega u_j^2\,dA=-\frac12
	\]
	since $u_j$ is real, we obtain the first identity. Similarly,
	\[
		\sum_k(\lambda_k-\lambda_j)\,d_{jk}=\|(\Delta-\lambda_j)(zu_j)\|^2=16\|\pbar u_j\|^2=4\lambda_j,
	\]
	and adding $2\lambda_j\sum_kd_{jk}=4\lambda_j$ gives the second identity.
\end{proof}

We apply this result to estimate the partial trace of $C$.

\begin{lemma}\label[lemma]{lem:partial}
	For every $j\geq1$, we have $\lambda_j<4m_j$ and
	\begin{equation}
		\sum_{k\leq j}a_k\geq2j\Bigl(\frac{1}{\lambda_j}+\frac{1}{4m_j-\lambda_j}\Bigr).
	\end{equation}
\end{lemma}

\begin{proof}
	Fix $j\geq 1$. Consider \cref{eq:h} and the two identities in \cref{eq:rules} with $(j,k)$ replaced by $(k,\ell)$, and sum over $k\leq j$. By antisymmetry of $d_{k\ell}$, the terms with $\ell\leq j$ cancel in pairs. Writing $\sum$ for the sum over $k\leq j<\ell$, we obtain
	\[
		\sum_{k\leq j}a_k=\sum\Bigl(\frac{1}{\lambda_\ell}+\frac{1}{\lambda_k}\Bigr)d_{k\ell},\quad
		\sum d_{k\ell}=2j,\quad\textrm{and}\quad
		\sum(\lambda_\ell+\lambda_k)\,d_{k\ell}=8jm_j.
	\]
	Since $d_{k\ell}\geq0$ over the range of the sums, the weighted averages
	\[
		\alpha=\frac{1}{2j}\sum\lambda_kd_{k\ell}
		\qquad\textrm{and}\qquad
		\beta=\frac{1}{2j}\sum\lambda_\ell d_{k\ell}
	\]
	satisfy $0<\alpha\leq\lambda_j\leq\beta$ and $\alpha+\beta=4m_j$. In particular, $4m_j-\lambda_j\geq\alpha>0$. Moreover,
	\[
		\lambda_j(4m_j-\lambda_j)-\alpha\beta
		=(\lambda_j-\alpha)(\beta-\lambda_j)\geq0.
	\]
	Combining this with Cauchy--Schwarz gives
	\[
		\sum_{k\leq j}a_k
		\geq2j\Bigl(\frac1\alpha+\frac1\beta\Bigr)
		=\frac{8jm_j}{\alpha\beta}
		\geq2j\Bigl(\frac{1}{\lambda_j}+\frac{1}{4m_j-\lambda_j}\Bigr),
	\]
	as claimed.
\end{proof}

Finally, we may now prove \cref{prop}.

\begin{proof}[Proof of \cref{prop}]
	Fix $\Lambda>0$ and write $t_+=\max\{t,0\}$. By \cref{eq:h,eq:rules}, we have
	\[
		\sum_k(\lambda_k-\Lambda)
		\Bigl(\frac1{\lambda_k}+\frac1{\lambda_j}\Bigr)d_{jk}
		=8-\Lambda a_j.
	\]
	Multiplying by $(\Lambda-\lambda_j)_+$ and summing over $j$, we obtain
	\[
		\sum_j(\Lambda-\lambda_j)_+(8-\Lambda a_j)
		=\sum_{j,k}(\Lambda-\lambda_j)_+(\lambda_k-\Lambda)
		\Bigl(\frac1{\lambda_k}+\frac1{\lambda_j}\Bigr)d_{jk}
		\geq0.
	\]
	Only finitely many $j$ contribute to these sums. The terms with $\lambda_j,\lambda_k<\Lambda$ cancel in pairs, since their coefficient is symmetric in $j,k$ while $d_{jk}$ is antisymmetric. The remaining nonzero terms have $\lambda_j<\Lambda<\lambda_k$ and are nonnegative. Consequently,
	\[
		\sum_j\Bigl(1-\frac{\lambda_j}{\Lambda}\Bigr)_+a_j
		\leq\frac{8}{\Lambda^2}
		\sum_j(\Lambda-\lambda_j)_+.
	\]
    
	Let $N(t)=\#\{j:\lambda_j<t\}$ be the eigenvalue counting function. Weyl's law \cite[Theorem~3.3.1 and Remark~3.3.2]{LMP} gives
	\[
		N(t)=\frac{A}{4\pi}t+o(t)
		\qquad\text{as }t\to\infty.
	\]
	Integrating this asymptotic, we obtain
	\[
		\sum_j(\Lambda-\lambda_j)_+
		=\int_0^\Lambda N(t)\,dt
		=\frac{A}{8\pi}\Lambda^2+o(\Lambda^2)
		\qquad\text{as }\Lambda\to\infty.
	\]
	Since $a_j\geq0$ by \cref{lem:positive} and $(1-\lambda_j/\Lambda)_+$ increases to $1$ as $\Lambda\to\infty$, monotone convergence yields
	\[
		\sum_ja_j
		=\lim_{\Lambda\to\infty}
		\sum_j\Bigl(1-\frac{\lambda_j}{\Lambda}\Bigr)_+a_j
		\leq\lim_{\Lambda\to\infty}
		\frac{8}{\Lambda^2}\sum_j(\Lambda-\lambda_j)_+
		=\frac{A}{\pi}.
	\]
	The nonnegative operator $C$ is therefore trace class, with
	$\Tr C=\sum_ja_j\leq A/\pi$.

	For the reverse inequality, \cref{lem:partial} and $1/a+1/b\geq4/(a+b)$ give
	\[
		\Tr C\geq\sum_{k\leq j}a_k
		\geq2j\Bigl(\frac1{\lambda_j}
		+\frac1{4m_j-\lambda_j}\Bigr)
		\geq\frac{2j}{m_j}.
	\]
	Letting $j\to\infty$ and applying Weyl's law once more gives $\Tr C\geq A/\pi$, completing the proof.
\end{proof}

The main theorem now follows immediately.

\begin{proof}[Proof of \cref{thm}]
	Combine \cref{prop,lem:partial} to deduce
	\[
		\frac{8jm_j}{\lambda_j(4m_j-\lambda_j)}
		=2j\Bigl(\frac{1}{\lambda_j}+\frac{1}{4m_j-\lambda_j}\Bigr)
		\leq\sum_{k\leq j}a_k\leq\frac{A}{\pi}.
	\]
	As $4m_j-\lambda_j>0$, this rearranges to \cref{eq:evalue_bound}. Note that all arguments provided here require no regularity assumption on the boundary.
\end{proof}

\begin{remark}\label[remark]{rem:bergman}
	We can alternatively understand the trace of $C$ from the perspective of Bergman spaces.  Recall that the Bergman space $A^2(\Omega)$ is defined by
	\[
		A^2(\Omega)=\{f\in L^2(\Omega)\textrm{ with }\pbar f=0\}.
	\]
	In other words, $A^2(\Omega)$ consists of the square-integrable holomorphic functions on $\Omega$. Denote by $V=\{\overline{f}\,\colon f\in A^2(\Omega)\}$ the subspace of $L^2(\Omega)$ consisting of antiholomorphic functions, and by $W$ the orthogonal complement to $V$. One checks that $U\colon \mathcal{H}\to W$ given by $Uf=2\pbar f$ is unitary. Furthermore, since $\pbar(zu)=z\pbar u$, the operator $U$ intertwines $M$ on $\mathcal{H}$ with multiplication by $z$ on $W$. With $N$ denoting multiplication by $\overline z$ on $V$ and $B\colon W\to V$ by $Bg=\operatorname{Proj}_V(\bar zg)$, we obtain
	\[
		UCU^{-1}=B^*B
		\qquad\text{and}\qquad
		[N^*,N]=BB^*.
	\]
	The trace formula of Berger and Shaw \cite[Theorem~7]{BergerShaw} now yields
	\[
		\Tr C=\Tr(B^*B)=\Tr(BB^*)=\Tr[N^*,N]=\frac{A}{\pi},
	\]
	as desired. The proof of \cref{prop} given above avoids this operator-theoretic machinery and proceeds instead through spectral estimates.
\end{remark}

\begin{remark}
	The identities in \cref{eq:rules} arise in the eigenvalue inequalities of Payne--P\'olya--Weinberger, Hile--Protter, and Yang. Since the functions $u_j$ are real valued, we may write $|c_{jk}|^2=\langle xu_j,u_k\rangle^2+\langle yu_j,u_k\rangle^2$, which corresponds to the $A_{jk}$ coefficients in \cite[(2.19)]{Ashbaugh}. The same antisymmetry cancellation as above appears in the derivation of Yang's bound $\lambda_{k+1}\leq3m_k$ \cite[(2.28)--(2.29)]{Ashbaugh}. Several of these identities may be understood from the more abstract framework presented in \cite[Theorem~2.2]{LevitinParnovski}.
\end{remark}

\begin{acknowledgements}
	The author is grateful for the comments and advice provided by Otis Chodosh and Rafe Mazzeo. This work was carried out with assistance from several AI tools that contributed through mathematical exploration, technical discussions, and editorial assistance. All AI-assisted proofs were independently checked and understood by the author, who devoted substantial effort to simplifying the arguments and improving the exposition. The author bears responsibility for this manuscript's content and correctness.
\end{acknowledgements}

\end{document}